\documentclass{amsart}
\usepackage{amsfonts, amssymb, amsmath, eucal, verbatim, amsthm, amscd, enumerate}
\usepackage{graphicx}
\newtheorem{theorem}{Theorem}
\newtheorem{lemma}[theorem]{Lemma}

\theoremstyle{definition}

\theoremstyle{remark}

\def\[{{\langle}}
\def\]{{\rangle}}
\def\R{{\mathbb  R}}

\def\Z{{\mathbb Z}}
\def\N{{\mathbb N}}

\begin{document}
\title{Heat-flow monotonicity related to the Hausdorff--Young inequality}
\author{Jonathan Bennett}
\author{Neal Bez}
\author{Anthony Carbery}
\thanks{The first and second authors were supported by EPSRC grant
EP/E022340/1. The third author would like to thank William Beckner
for many useful and interesting conversations on the sharp
Hausdorff--Young inequality.}
\address{Jonathan Bennett and Neal Bez\\ School of Mathematics \\
The University of Birmingham \\
The Watson Building \\
Edgbaston \\
Birmingham \\
B15 2TT \\
United Kingdom}\email{J.Bennett@bham.ac.uk\\ N.Bez@bham.ac.uk}
\address{Anthony Carbery\\ School of Mathematics and Maxwell Institute for Mathematical Sciences\\
The University of Edinburgh \\
James Clerk Maxwell Building \\
King's Buildings \\
Edinburgh \\
EH3 9JZ \\
United Kingdom} \email{A.Carbery@ed.ac.uk}
\date{26th of June 2008}
\begin{abstract}
It is known that if $q$ is an even integer then the
$L^q(\mathbb{R}^d)$ norm of the Fourier transform of a superposition
of translates of a fixed gaussian is monotone increasing as their
centres ``simultaneously slide" to the origin. We provide explicit
examples to show that this monotonicity property fails dramatically
if $q > 2$ is not an even integer. These results are equivalent,
upon rescaling, to similar statements involving solutions to heat
equations. Such considerations are natural given the celebrated
theorem of Beckner concerning the gaussian extremisability of the
Hausdorff--Young inequality.
\end{abstract}
\maketitle For $d \in \N$ we let $H_t$ denote the heat kernel on
$\R^d$ given by
\begin{equation*}
H_t(x)=t^{-d/2}e^{-\pi|x|^2/t},
\end{equation*}
and we define the Fourier transform $\widehat{\mu}$ of a finite
Borel measure $\mu$ on $\R^d$ by
\begin{equation*}
  \widehat{\mu}(\xi) := \int_{\R^d} e^{-2\pi i x \cdot \xi}\,d\mu(x).
\end{equation*}
In what follows, for $p \in [1,\infty]$, $p'$ will denote the dual
exponent satisfying $\tfrac{1}{p} + \tfrac{1}{p'} = 1$. For $\mu$ a
positive finite Borel measure on $\mathbb{R}^d$ and $2 \leq q \leq
p' \leq \infty$, let
  $Q_{p,q}:(0,\infty)\rightarrow\mathbb{R}$ be given by
  \begin{equation*}
    Q_{p,q}(t) = t^{\frac{d}{2}\left(\frac{1}{q}-\frac{1}{p'} \right)}
    \left\| \widehat{u(t,\cdot)^{1/p}} \right\|_q,
  \end{equation*}
  where
  $u(t,\cdot) = H_t \ast \mu$.
If $q=2k$ is an even integer then by Plancherel's theorem one may
write $Q_{p,q}$ in terms of a $k$-fold convolution
\begin{equation}\label{James}
Q_{p,q}(t) = t^{\frac{d}{2}\left(\frac{1}{q}-\frac{1}{p'}
\right)}\|u(t,\cdot)^{1/p}*\cdots *u(t,\cdot)^{1/p}\|_{2}^{2/q}.
\end{equation}
Expressions of this type are by now well-known to be nondecreasing
for $t>0$ and this follows from the heat-flow approach to
generalised Young's inequalities developed in \cite{CLL} and \cite{BCCT} (see also
\cite{BB} for an alternative approach). For the convenience of the reader, in the Appendix we
have included a sketch of how this monotonicity follows from
\cite{BCCT}. We note that for $p=1$ this is a particularly
straightforward exercise using the fact that
$\widehat{H_t}(\xi)=e^{-\pi t|\xi|^2}$. The purpose of this article
is to show that this heat-flow monotonicity fails dramatically if
$q$ is not an even integer.
\begin{theorem} \label{t:counter}
  Let $d \in \N$, $2\leq q\leq p' \leq \infty$ and suppose $q$ is not an even integer. Then there exists
  a positive finite Borel measure $\mu$ on $\R^d$ such that if $u(t,\cdot) = H_t \ast
  \mu$ then
\begin{equation*} Q_{p,q}(t) := t^{\frac{d}{2}\left(\frac{1}{q}-\frac{1}{p'} \right)}
\left\| \widehat{u(t,\cdot)^{1/p}} \right\|_q
\end{equation*}
  is strictly decreasing for sufficiently small $t > 0$.
\end{theorem}
By making an appropriate rescaling one may rephrase the above
results in terms of ``sliding gaussians" in the following way. Let
$\mu$ be a positive finite Borel measure on $\mathbb{R}^d$, and
define $f:(0,\infty)\times\mathbb{R}^d\rightarrow\mathbb{R}$ by
$$
f(t,x)=\int_{\mathbb{R}^d}e^{-\pi|x-tv|^2}d\mu(v).$$ We interpret
$f$ as a superposition of translates of a fixed gaussian, which
simultaneously slide to the origin as $t$ tends to zero. For $2 \leq
q \leq p'\leq\infty$ define the quantity $\widetilde{Q}_{p,q}(t)$ by
$$ \widetilde{Q}_{p,q}(t)=\|\widehat{f(t,\cdot)^{1/p}}\|_q.$$ The
nondecreasingness of $Q_{p,q}$ for $q$ an even integer tells us that
$\widetilde{Q}_{p,q}(t)$ is nonincreasing, and Theorem
\ref{t:counter} tells us that whenever $q$ is not an even integer,
there exist such measures $\mu$ for which $\widetilde{Q}_{p,q}(t)$
is \emph{strictly increasing} for sufficiently large $t$. It is
interesting to note that the quantity $\|f(t,\cdot)\|_{q'/p}^{1/p}$,
related to $\widetilde{Q}_{p,q}(t)$ via the Hausdorff--Young
inequality
$$
\widetilde{Q}_{p,q}(t)\leq\|f(t,\cdot)^{1/p}\|_{q'}=\|f(t,\cdot)\|_{q'/p}^{1/p},$$
is nonincreasing \emph{for all} $2\leq q\leq p'\leq\infty$, whether
$q$ is an even integer or not. See \cite{BCCT}.

The quantities $Q_{p,q}$ have a more direct relation with the
Hausdorff--Young inequality when $q=p'$. Suppose that $d\mu(x)=|f(x)|^pdx$ for some sufficiently well-behaved function $f$ on $\mathbb{R}^d$ (such as bounded with compact support). In this case, if $Q_{p,q}$ is nondecreasing then it is straightforward to verify that
\begin{equation*}
 \|\,\widehat{|f|}\,\|_{p'}=\lim_{t \rightarrow 0} Q_{p,q}(t) \leq \lim_{t \rightarrow \infty}
 Q_{p,q}(t)=\Bigl\|\widehat{H_1^{1/p}}\Bigr\|_{p'}\|f\|_{p},
\end{equation*}
where $H_1$ is the heat kernel at time $t=1$.
Now if $p'$ is an even integer then
$$
\|\,\widehat{f}\,\|_{p'}\leq\|\,\widehat{|f|}\,\|_{p'},
$$
and so one recovers
the sharp form of the Hausdorff--Young inequality on $\R^d$
\begin{equation} \label{e:HY}
\|\widehat{f}\:\|_{p'}\leq
\left(\frac{p^{1/p}}{p'^{1/p'}}\right)^{d/2}\|f\|_p
\end{equation}
for $p'$ an even integer, due to Babenko \cite{Babenko1},
\cite{Babenko2}. We note that since \eqref{e:HY} is not in general valid for nonnegative $f$ when $p' < 2$, it follows that $Q_{p,q}(t)$ cannot possibly be
nondecreasing for $t > 0$ when $q=p'<2$.

Theorem \ref{t:counter} is of course a significant obstacle to
finding a proof based on heat-flow of the sharp Hausdorff--Young
inequality due to Beckner \cite{Beckner2}, \cite{Beckner}; i.e. for
\emph{all} $p'\in [2,\infty)$. It should also be remarked that
whenever $p' \in [2,\infty)$ is not an even integer there exists
$f\in\mathcal{S}(\mathbb{R}^d)$ such that
\begin{equation} \label{e:maj}
\|\widehat{f}\|_{p'} > \| \, \widehat{|f|} \,\|_{p'}.
\end{equation}
Thus, in general, it is not without loss of generality that one
considers nonnegative inital data for the heat-flow. Inequality
\eqref{e:maj} may be seen as a consequence of an observation due to
Hardy and Littlewood \cite{HL} concerning a majorant problem in the
context of classical Fourier series. In fact, the counterexamples in
our proof of Theorem \ref{t:counter} are somewhat in the spirit of
the Hardy--Littlewood majorant counterexample in \cite{HL}.

The idea of looking for monotone quantities underlying inequalities
in analysis is of course not new. One way of constructing such
quantities which has been successful in recent years is via
heat-flow methods of the type we consider here. As we have already
mentioned, this heat-flow perspective applies to a wide variety of
so-called generalised Young's inequalities (or Brascamp--Lieb
inequalities), which include the classical Young's convolution,
multilinear H\"older, and Loomis--Whitney inequalities -- see in particular
\cite{CLL} and \cite{BCCT}. Among other notable (and closely
related) examples from harmonic analysis are certain multilinear
analogues of Kakeya maximal inequalities \cite{BCT} and adjoint
restriction inequalities for the Fourier transform -- see the
forthcoming \cite{BCCH}.
\section*{Proof of Theorem \ref{t:counter}}
It suffices to handle $d=1$, since if $\mu$ is a one-dimensional
counterexample, then its $d$-fold tensor product is a
$d$-dimensional counterexample. The case $p=1$ will turn out to be
pivotal and so we deal with that first of all.

Observe that if $\mu$ is a finite sum of Dirac delta measures each
supported at an integer, then $\widehat{\mu}$ is a trigonometric
polynomial, and thus a bounded periodic function on $\R$ with period
$1$. If $c_n$ denotes the $n$th Fourier coefficient of
$|\widehat{\mu}|^q$, then
  \begin{eqnarray*} Q_{1,q}(t)^q & = & t^{1/2} \int_\R \widehat{H_t}(\xi)^q
|\widehat{\mu}(\xi)|^q\,d\xi \\
& = & \sum_{n \in \Z} c_n t^{1/2} \int_\R \widehat{H_t}(\xi)^q
 e^{2\pi in\xi}\,d\xi \\
& = & \sum_{n \in \Z} c_n t^{1/2} \int_\R e^{-q\pi t\xi^2}
 e^{2\pi in\xi}\,d\xi \\
& = & q^{-1/2} \sum_{n \in \Z} c_n e^{-\pi n^2/qt}.
\end{eqnarray*}
Since $q > 2$ it follows that $|\widehat{\mu}|^q$ is continuously
differentiable everywhere and thus the Fourier coefficients of
$|\widehat{\mu}|^q$ are absolutely summable. This is sufficient to
justify the above interchange of summation and integration.
Furthermore, note that
$$\sum_{\substack{n \in \Z \\ n \neq 0}} n^2c_n t^{-2} e^{-\pi n^2/qt}$$
is uniformly convergent because, trivially, each summand is bounded
in modulus by an absolute constant (i.e. independent of $t
> 0$ and $n \neq 0$) multiple of $1/n^2$. Again, by standard
results, it follows that we may differentiate the above expression
for $Q_{1,q}(t)^q$ term by term to get,
  \begin{eqnarray*}
\frac{d}{dt}[Q_{1,q}(t)^q] & = & \frac{\pi}{q^{3/2}}t^{-2}
\sum_{\substack{n \in \Z \\ n
\neq 0}} n^2c_n e^{-\pi n^2/qt} \\
 & = & \frac{\pi}{q^{3/2}}t^{-2} \sum_{n=1}^\infty n^2(c_n+c_{-n}) e^{-\pi n^2/qt} \\
& = &\frac{\pi}{q^{3/2}}t^{-2} e^{-\pi/qt} \left( c_1+c_{-1} +
\sum_{n=2}^\infty n^2(c_n+c_{-n})
 e^{-\pi(n^2-1)/qt} \right).
\end{eqnarray*}
  Since $(c_n)_{n \in \Z}$ is, in particular, a bounded sequence it follows that
  \begin{equation*}
  \sum_{n =2}^\infty n^2(c_n+c_{-n}) e^{-\pi(n^2-1)/qt} \longrightarrow 0
\end{equation*}
  as $t$ tends to zero. Thus, to prove Theorem \ref{t:counter} when $p=1$
  it suffices to find a $\mu$ formed out of a finite sum of Dirac delta measures
  each supported at an integer and such that
  \begin{equation} \label{e:negcoeff} c_1+c_{-1} < 0.
\end{equation}
To this end, we let $m,n \in \Z$ be coprime, $r \in (0,1/2)$ and
\begin{equation}\label{example}
\mu := \delta_0 + r\delta_m + r\delta_n,
\end{equation}
so that
\begin{equation*}
\widehat{\mu}(\xi) = 1 + re^{-2\pi im \xi} + re^{-2\pi in \xi}.
\end{equation*}
Since $|\widehat{\mu}(\xi)|^2 =
\widehat{\mu}(\xi)\overline{\widehat{\mu}(\xi)}$ we have that
\begin{equation*}
  |\widehat{\mu}(\xi)|^q = \sum_{k=0}^\infty a_kr^k(e^{-2\pi
  im\xi}+e^{-2\pi in\xi})^k \sum_{k'=0}^\infty a_{k'}r^{k'}(e^{2\pi
  im\xi}+e^{2\pi in\xi})^{k'}
\end{equation*}
where $a_k$ is the $k$th binomial coefficient in the expansion of
$(1+x)^{q/2}$; i.e.
\begin{equation*}
a_k = \frac{\frac{q}{2}(\frac{q}{2}-1)\ldots(\frac{q}{2}-k+1)}{k!}.
\end{equation*}
Observe that if $k < q/2+1$ then $a_k > 0$, and thereafter $a_k$ is
strictly alternating in sign. Now,
\begin{eqnarray*}
  c_1+c_{-1} & = & \int_0^1 |\widehat{\mu}(\xi)|^q(e^{-2\pi i \xi} + e^{2\pi i\xi})\,d\xi \\
  & = & \sum_{k,k'=0}^\infty a_k a_{k'} r^{k+k'} \\
  && \quad \times
  \int_0^1 (e^{-2\pi im\xi}+e^{-2\pi in\xi})^k (e^{2\pi im\xi}+e^{2\pi
  in\xi})^{k'}(e^{-2\pi i \xi}+e^{2\pi i \xi})\,d\xi
\end{eqnarray*}
(of course, since $\mu$ is a real measure it follows that
$|\widehat{\mu}|$ is even and therefore $c_1 = c_{-1}$; nevertheless
it is slightly more convenient to consider $c_1+c_{-1}$ in order to
preserve a certain symmetry later in the proof). To justify the
above interchange of summation and integration, it suffices to show
that $\sum_{k \geq 0} |a_k|(2r)^k$ is finite. Since $(a_k)_{k \geq
0}$ is a bounded sequence and $r \in (0,1/2)$ this is immediate.
Therefore,
\begin{equation*}
  c_1+c_{-1} = \sum_{k,k'=0}^\infty a_k a_{k'} r^{k+k'}
  \sum_{(\textbf{j},\textbf{j}') \in \Lambda_{k,k'}} {k \choose j_1}
  {k' \choose j_1'}
\end{equation*}
where
\begin{align*}
\Lambda_{k,k'} := \{ (\textbf{j},\textbf{j}') =
((j_1,j_2),(j_1',j_2')) \in (\N_0^2)^2 \,\, : \,\, & j_1+j_2 = k,
\,\, j_1'+j_2' = k' \,\, \mbox{and} \\ & \qquad
m(j_1-j_1')+n(j_2-j_2')=\pm1\}
\end{align*}
and $\N_0 := \N \cup \{0\}$.

We claim that by choosing $m$ and $n$ appropriately (depending on
$q$) we can ensure that $\Lambda_{k,k'}$ is empty \emph{whenever}
$a_ka_{k'} > 0$. It will only remain to check that $\Lambda_{k,k'}$
is nonempty for some $k$ and $k'$ for which $a_ka_{k'} < 0$. The
proof of the claim proceeds as follows. Firstly, a simple argument
will show that if $n-m$ is even the sets $\Lambda_{k,k'}$ are empty
whenever $k$ and $k'$ have the same parity. A second argument will
show that $\Lambda_{k,k'}$ is empty whenever one of $k$ and $k'$ is
less than $q/2+1$ upon an appropriate choice of $m$ and $n$. This
leaves a contribution from summands with $k$ and $k'$ greater than
$q/2+1$ and, as long as one summand is nonzero, it is clear that
$c_1+c_{-1} < 0$ as required.

We now turn to the details. Since $m$ and $n$ are coprime there
exist integers $\alpha_0$ and $\beta_0$ such that
\begin{equation} \label{e:a0b0}
\alpha_0m+\beta_0n = 1;
\end{equation}
moreover if $\alpha m+\beta n = \pm1$ for integers $\alpha$ and
$\beta$ then
\begin{equation*}
(\alpha,\beta) = \pm(\alpha_0,\beta_0) + N(n,-m) \quad \text{for
some $N \in \Z$}.
\end{equation*}
Therefore, if $(\textbf{j},\textbf{j}') \in \Lambda_{k,k'}$ then
\begin{equation}
\label{e:j1j1'} j_1-j_1' = \pm\alpha_0 + Nn
\end{equation}
and
\begin{equation}
\label{e:j2j2'} j_2-j_2' = \pm\beta_0 - Nm
\end{equation}
for some $N \in \Z$,
\begin{equation} \label{e:j1j2}
j_1+j_2 = k,
\end{equation}
and
\begin{equation} \label{e:j1'j2'}
j_1'+j_2' = k'.
\end{equation}

\begin{lemma} \label{l:parity}
  Suppose $n-m$ is even, and that $k$ and $k'$ have the same parity. Then $\Lambda_{k,k'}$ is
  empty.
\end{lemma}
\begin{proof}
Let $(\textbf{j},\textbf{j}') \in \Lambda_{k,k'}$. By summing
equations \eqref{e:j1j1'}, \eqref{e:j2j2'}, \eqref{e:j1j2}, and
\eqref{e:j1'j2'} it follows that
\begin{equation} \label{e:incon}
 2(j_1+j_2) = \pm(\alpha_0 + \beta_0) + (n-m)N + k+k'.
\end{equation}
So $n-m$ even implies $\alpha_0 + \beta_0$ is even. On the other
hand, $n-m$ even and
\begin{equation*}
 (\alpha_0 + \beta_0)m + (n-m)\beta_0 = 1
\end{equation*}
imply that $\alpha_0 + \beta_0$ is odd. Hence $\Lambda_{k,k'} =
\emptyset$.
\end{proof}
For fixed integers $\alpha_0$ and $\beta_0$ satisfying
\eqref{e:a0b0}, define
\begin{equation*}
  \alpha_\ast := \min\{|\alpha_0+Nn|:N \in \Z\}
\end{equation*}
and
\begin{equation*}
  \beta_\ast := \min \{|\beta_0-Nm| : N \in
  \Z\}.
\end{equation*}
\begin{lemma} \label{l:star}
  Suppose $m$ and $n$ are positive integers. Then the set $\Lambda_{k,k'}$ is empty whenever
  \begin{equation} \label{e:tocheck}
    k' \geq 0 \quad \mbox{and} \quad 0 \leq k \leq \min\{\alpha_\ast,\beta_\ast\} - 1
  \end{equation}
  or
  \begin{equation*}
    k \geq 0 \quad \mbox{and} \quad 0 \leq k' \leq \min\{\alpha_\ast,\beta_\ast\} -
    1.
  \end{equation*}
\end{lemma}
\begin{proof}
  By symmetry, it suffices to check that $\Lambda_{k,k'}$ is empty when
  \eqref{e:tocheck} holds.
  Suppose $(\textbf{j},\textbf{j}') \in \Lambda_{k,k'}$ and set
  $(\alpha,\beta) := (j_1-j_1',j_2-j_2')$ so that $\alpha m +\beta n = \pm 1$. By \eqref{e:j1j1'},
  \eqref{e:j2j2'}, \eqref{e:j1j2} and \eqref{e:j1'j2'} it follows that
  \begin{equation*}
    k \geq j_1 \geq j_1-j_1' = \pm \alpha_0 +Nn
  \end{equation*}
  and
  \begin{equation*}
    k \geq j_2 \geq j_2-j_2' = \pm \beta_0 -Nm
  \end{equation*}
  for some $N \in \Z$. Since $m(j_1-j_1')+n(j_2-j_2') = \pm 1$, we have that $\pm \alpha_0+Nn$ and $\pm
  \beta_0-Nm$ must have opposing signs. Therefore,
  \begin{equation*}
    k  \geq  \min_{N \in \Z} \min \{ |\alpha_0 + Nn|,|\beta_0-Nm|
    \} = \min \{\alpha_\ast,\beta_\ast\}.
  \end{equation*}
  Hence $\Lambda_{k,k'}$ is empty when \eqref{e:tocheck} holds.
\end{proof}
To complete the proof of Theorem \ref{t:counter} when $p=1$, let
$k(q)$ denote the smallest integer greater than $q/2+1$ and consider
the following particularly simple choice of $m$ and $n$:
\begin{equation*} \label{e:choice}
m = 2k(q)+1 \quad \mbox{and} \quad n=m+2.
\end{equation*}
Now
\begin{equation} \label{e:star}
\left(\frac{m+1}{2}\right)m + \left(-\frac{m-1}{2}\right)n = 1
\end{equation}
and an easy calculation shows that
\begin{equation*}
\alpha_\ast = k(q)+1 \quad \mbox{and} \quad \beta_\ast = k(q).
\end{equation*}
Therefore, by Lemma \ref{l:parity} and Lemma \ref{l:star},
\begin{equation*}
  c_1+c_{-1} = \sum_{\substack{k,k'\geq k(q) \\ k,k'
  \text{opposing parity}}}
  a_ka_{k'}r^{k+k'} \sum_{(\textbf{j},\textbf{j}') \in \Lambda_{k,k'}} {k \choose j_1}
  {k' \choose j_1'}.
\end{equation*}
Moreover, \eqref{e:star} trivially implies that
$\Lambda_{k(q)+1,k(q)}$ is nonempty and hence $c_1+c_{-1} < 0$, as
required.

We remark that there are many choices of the integers $m$ and $n$
which would have worked when $p=1$. Our argument for $p > 1$ below
capitalises on the fact that there exist $m$ and $n$ separated by a
distance $O(m)$ and such that $\Lambda_{k,k'}$ continues to be empty
whenever $a_ka_{k'} > 0$. Moreover, we shall require that $m$ can be
chosen as large as we please. To see that such a choice of $m$ and
$n$ is possible, suppose
\begin{equation} \label{e:choice2}
m = 3k_0+1 \quad \mbox{and} \quad n=2m+3
\end{equation}
where $k_0$ is an even integer greater than $q/2+1$. A
straightforward computation shows that
\begin{equation} \label{e:star2}
\left(\frac{2m+1}{3}\right)m + \left(-\frac{m-1}{3}\right)n = 1
\end{equation}
and consequently
\begin{equation*}
\alpha_\ast = 2k_0+1 \quad \text{and} \quad \beta_\ast = k_0.
\end{equation*}
Furthermore, $\Lambda_{2k_0+1,k_0}$ is nonempty by \eqref{e:star2}.
Since $k_0$ is even it again follows from Lemma \ref{l:parity} and
Lemma \ref{l:star} that $c_1+c_{-1} < 0$. We emphasise that $k_0$
can be as large as we please in this argument.

Now suppose $p > 1$ and let $m$ and $n$ be given by
\eqref{e:choice2}. The idea behind the remainder of the proof is the
following. It is clear that for sufficiently large $m$ and small $t
> 0$, $H_t \ast \mu$ is a finite sum of ``well-separated"
gaussians causing $(H_t \ast \mu)^{1/p}$ to be ``very close" to
$H_t^{1/p}\ast \widetilde{\mu}$, where
\begin{equation*}
\widetilde{\mu} := \delta_0 + r^{1/p}\delta_m + r^{1/p}\delta_n.
\end{equation*}
Given this, $Q_{p,q}(t)^q$ should be ``very close" to
$$p^{dq/2}t^{d/2}\|\widehat{H_{pt} \ast \widetilde{\mu}}\|_q^q.$$
Furthermore, if $r \in (0,1/2^p)$, this last quantity, as we have
seen, is strictly decreasing for sufficiently small $t$ with
derivative bounded above in modulus by a constant multiple of
$t^{-2}e^{-\pi/pqt}$. Thus, to conclude our proof of Theorem
\ref{t:counter} when $p > 1$, it suffices to check that the error
\begin{eqnarray*}
\begin{aligned}
E_{p,q}(t)&:=Q_{p,q}(t)^q-p^{dq/2}t^{d/2}\|\widehat{H_{pt} \ast \widetilde{\mu}}\|_q^q\\
&=t^{\frac{1}{2}\left(1-\frac{q}{p'}\right)}\int_{\mathbb{R}}\bigg(\big|((H_t*\mu)^{\frac{1}{p}})\hat{\;}\big|^q
-\big|(H_t^{\frac{1}{p}}*\widetilde{\mu})\hat{\;}\big|^q\bigg)
\end{aligned}
\end{eqnarray*}
has a derivative bound of the form
\begin{equation}\label{perm} |E_{p,q}'(t)|\leq
Ct^{-\gamma}e^{-cm^2/t},\end{equation} for sufficiently large $m$
and small $t$. Here $C=C_{p,q,m}$ denotes a constant that may depend
on $p$, $q$ and $m$, and $\gamma=\gamma_{p,q}$ and $c=c_{p,q}$
constants that may depend on $p$ and $q$.

Differentiating and grouping terms we obtain
\begin{equation*}
  E_{p,q}'(t) = \tfrac{1}{2}\left(1-\tfrac{q}{p'}\right)t^{-\frac{1}{2}\left(1+\frac{q}{p'}\right)}(I + II) +
  t^{\frac{1}{2}\left(1-\frac{q}{p'}\right)}(III + IV),
\end{equation*}
where
\begin{align*}
I & :=
\int\big|((H_t*\mu)^{\frac{1}{p}})\hat{\;}\big|^{q-2}((H_t*\mu)^{\frac{1}{p}})\hat{\;}
\bigg(((H_t*\mu)^{\frac{1}{p}})\hat{\;}-(H_t^{\frac{1}{p}}*\widetilde{\mu})\hat{\;}\bigg),\\
II & := \int(H_t^{\frac{1}{p}}*\widetilde{\mu})\hat{\;}\bigg(\,
\big|((H_t*\mu)^{\frac{1}{p}})\hat{\;}\big|^{q-2}((H_t*\mu)^{\frac{1}{p}})\hat{\;}-
\big|(H_t^{\frac{1}{p}}*\widetilde{\mu})\hat{\;}\big|^{q-2}
(H_t^{\frac{1}{p}}*\widetilde{\mu})\hat{\;}\bigg),\\
III & :=
\int\big|((H_t*\mu)^{\frac{1}{p}})\hat{\;}\big|^{q-2}((H_t*\mu)^{\frac{1}{p}})\hat{\;}
\bigg((\partial_t(H_t*\mu)^{\frac{1}{p}})\hat{\;}-(\partial_tH_t^{\frac{1}{p}}*\widetilde{\mu})\hat{\;}\bigg),\\
IV & :=
\int(\partial_tH_t^{\frac{1}{p}}*\widetilde{\mu})\hat{\;}\bigg( \,
\big|((H_t*\mu)^{\frac{1}{p}})\hat{\;}\big|^{q-2}((H_t*\mu)^{\frac{1}{p}})\hat{\;}-
\big|(H_t^{\frac{1}{p}}*\widetilde{\mu})\hat{\;}\big|^{q-2}
(H_t^{\frac{1}{p}}*\widetilde{\mu})\hat{\;}\bigg).
\end{align*}
Bounds of the form \eqref{perm} are easily obtained for $I$, $II$,
$III$ and $IV$ by elementary estimates, such as the Cauchy--Schwarz
and Hausdorff--Young inequalities, along with pointwise estimates on
the heat kernel $H_t$. We illustrate this for the term $I$, leaving
the remaining terms to the reader. Applying the Cauchy--Schwarz
inequality, the Hausdorff--Young inequality and Plancherel's
theorem, we have
\begin{eqnarray*}\begin{aligned}
I&\lesssim
\bigg(\int\big|((H_t*\mu)^{\frac{1}{p}})\hat{\;}\big|^{2(q-1)}\bigg)^{1/2}
\bigg(\int\big|((H_t*\mu)^{\frac{1}{p}})\hat{\;}-(H_t^{\frac{1}{p}}*\widetilde{\mu})\hat{\;}\big|^2\bigg)^{1/2}\\
&\leq
\bigg(\int(H_t*\mu)^{\frac{2(q-1)}{(2q-3)p}}\bigg)^{\frac{2q-3}{2}}
\bigg(\int\big|(H_t*\mu)^{\frac{1}{p}}-H_t^{\frac{1}{p}}*\widetilde{\mu}\big|^2\bigg)^{1/2}.
\end{aligned}\end{eqnarray*}
Here we have used the fact that $q\geq 2$. Now, the first integral
factor above is bounded by a power of $t$ (which is permissible).
For the second integral factor we split the integration over $\R$
into $\bigcup_0^6 I_j$, where
\begin{align*}
& I_0 = (-\infty,-\epsilon m], \quad I_1 = [-\epsilon m,\epsilon m],
\quad I_2 = [\epsilon m, (1-\epsilon)m], \quad
I_3 = [(1-\epsilon)m,(1+\epsilon)m], \\
& I_4 = [(1+\epsilon)m,(1-\epsilon)n], \quad I_5 =
[(1-\epsilon)n,(1+\epsilon)n] \quad \mbox{and} \quad I_6 =
[(1+\epsilon)n,\infty),
\end{align*}
for some sufficiently small positive absolute constant $\epsilon$.
We claim that a bound of the form \eqref{perm} holds for each term
$$
\int_{I_j}\bigg|(H_t*\mu)^{\frac{1}{p}}-H_t^{\frac{1}{p}}*\widetilde{\mu}\bigg|^2.
$$
For $j=0,2,4,6$ this is a simple consequence of the triangle
inequality combined with elementary estimates on the heat kernel
$H_t$. For $j=1$ this follows from the facts that for $x\in I_1$,
$$
H_t*\mu(x)=H_t(x)+O(t^{-1/2}e^{-cm^2/t})
$$
and
$$
H_t^{1/p}*\widetilde{\mu}(x)=H_t(x)^{1/p}+O(t^{-1/2p}e^{-cm^2/t}),$$
and the mean value theorem applied to the function $x\mapsto
x^{1/p}$. The cases $j=3$ and $j=5$ are similar.

\section*{Appendix: The even integer case} We will appeal to the following general theorem from \cite{BCCT}
(see Proposition 8.9).
\begin{theorem}\label{bl}
Let $m,n\in\mathbb{N}$, $n_1,\hdots,n_m\in\mathbb{N}$ and
$p_1,\hdots,p_m>0$. Suppose that for each $1\leq j\leq m$ there are
linear surjections $B_j:\mathbb{R}^{n}\rightarrow\mathbb{R}^{n_j}$
and $A_j:\mathbb{R}^{n_j}\rightarrow\mathbb{R}^{n_j}$ such that the
mapping $M=\sum_{j=1}^m\frac{1}{p_j}B_j^{*}A_jB_j$ is invertible and
$$B_jM^{-1}B_j^{*}\leq A_j^{-1}$$
for all $1\leq j\leq m$. Also, for each $1\leq j\leq m$ let $u_j$ be
a solution to the heat equation
\begin{equation}\label{he}
\partial_t u_j=\frac{1}{4\pi}\emph{div}(A_j^{-1}\nabla u_j).
\end{equation}
Then the quantity
$$
t^{\frac{1}{2}(\sum_{j=1}^m
\frac{n_j}{p_j}-n)}\int_{\mathbb{R}^n}\prod_{j=1}^m u_j(t,B_j
x)^{1/p_j}dx$$ is non-decreasing for $t > 0$.
\end{theorem}
Multiplying out the $L^2$ norm in \eqref{James}, we see that
\begin{equation*}
Q_{p,2k}(t)^{2k} = t^{\frac{1}{2}(\sum_{j=1}^m
\frac{n_j}{p_j}-n)}\int_{\mathbb{R}^n}\prod_{j=1}^m u_j(t,B_j
x)^{1/p_j}dx, \end{equation*} where $n=(2k-1)d$, $m=2k$, $n_j=d$,
$p_j=p$ and
\begin{eqnarray*}
B_j : & \R^{(2k-1)d} & \longrightarrow  \quad \R^d \\
& (x_1,\ldots,x_{2k-1}) & \longmapsto \quad \left\{
\begin{array}{llllll} x_j & \text{for $j=1,\ldots,2k-1$} \\
\sum_{j=1}^k x_j - \sum_{j'=k+1}^{2k-1} x_{j'} & \text{for $j=2k$}
\end{array}  \right.
\end{eqnarray*}
Furthermore, for each $j=1,\ldots,2k$, $u_j = u$ satisfies the heat
equation \eqref{he} where $A_j$ is the identity mapping. Since
$p_j=p$ for all $j$, by homogeneity it suffices to verify the
remaining hypotheses of Theorem \ref{bl} when $p=(2k)'$.

It is straightforward to verify that the matrix $M = \frac{1}{(2k)'}
\sum_{j=1}^{2k} B_j^*B_j$ has the following block form
representation with respect to the canonical bases.
\begin{equation*}
  M = \frac{1}{(2k)'} \left(\begin{array}{cccccccccccc}
\textbf{1}(k,k) + I_{kd} &
  -\textbf{1}(k,k-1) \\
  -\textbf{1}(k-1,k) & \textbf{1}(k-1,k-1) + I_{(k-1)d} \end{array}
  \right)
\end{equation*}
where $\textbf{1}(r,s)$ denotes the $rd$ by $sd$ matrix given by
\begin{equation*}
  \textbf{1}(r,s) = \left(\begin{array}{cccccccccccc}
I_d & I_d & \cdots & I_d \\
\vdots & \vdots & & \vdots \\
I_d & I_d & \cdots & I_d
\end{array} \right)
\end{equation*}
and $I_l$ is the $l$ by $l$ identity matrix. A direct computation
shows that $M$ is invertible with
\begin{equation*}
  M^{-1} = \frac{1}{2k-1} \left(\begin{array}{cccccccccccc}  -\textbf{1}(k,k) + 2kI_{kd} &
  \textbf{1}(k,k-1) \\
  \textbf{1}(k-1,k) & -\textbf{1}(k-1,k-1) + 2kI_{(k-1)d} \end{array}
  \right)
\end{equation*}
and $B_jM^{-1}B_j^* = I_d$ for each $j=1,\ldots,k$. It now follows
from Theorem \ref{bl} that $Q_{p,2k}(t)$ is nondecreasing for each
$t > 0$.
\bibliographystyle{amsalpha}

\end{document}